\documentclass[12pt]{article}
\usepackage{amsmath}
\usepackage{amsfonts}
\usepackage{amssymb}
\usepackage{amsthm}
\usepackage{mathrsfs}
\usepackage{amsxtra}
\usepackage[all,cmtip]{xy}
\usepackage{graphicx}
\usepackage{color}
\usepackage{comment}
\usepackage{verbatim}
\usepackage{url}

\newtheorem{thm}{Theorem}[section]

\newtheorem{lem}[thm]{Lemma}
\newtheorem{prop}[thm]{Proposition}

\newtheorem{cor}[thm]{Corollary}

\theoremstyle{definition}
\newtheorem{defn}[thm]{Definition}
\newtheorem{rem}[thm]{Remark}

\def\XXint#1#2#3{{\setbox0=\hbox{$#1{#2#3}{\int}$}
\vcenter{\hbox{$#2#3$}}\kern-.5\wd0}}

\def\bA{\mathbb{A}}
\def\cA{{\cal A}}

\def\cB{{\cal B}}
\def\cC{{\cal C}}

\def\cK{{\cal K}}

\def\cS{{\cal S}}

\def\cP{\mathcal{P}}

\def\cL{{\cal L}}
\def\cM{{\cal M}}
\def\cO{\mathcal{O}}

\def\cX{{\cal X}}
\def\Z{{\Bbb Z}}
\def\R{{\Bbb R}}

\def\C{{\Bbb C}}
\def\N{{\Bbb N}}
\def\Q{{\Bbb Q}}
\def\F{{\Bbb F}}

\def\Hom{\text{Hom}}

\def\det{\text{det}\,}

\def\GL{\operatorname{GL}}
\def\SL{\operatorname{SL}}

\def\Real{\operatorname{Re}}

\def\GL{\operatorname{GL}}

\def\sign{\operatorname{sign}}

\def\br{\mathbf{r}}

\def\F{{\Bbb F}}

\def\cP{\mathcal{P}}
\def\Norm{\operatorname{N}}

\def\Coeff{\operatorname{Coeff}}

\def\Frac{\operatorname{Frac}}

\begin{document}
\hoffset=-.4in
\voffset=-0.05in


 \title{The $p$-adic Shintani cocycle}
\author{G. Ander Steele}
\date{September 21, 2012}
\maketitle

\abstract{The Shintani cocycle on $\GL_n(\Q)$, as constructed by Hill, gives a cohomological interpretation of special values of zeta functions for totally real fields of degree $n$. We give an explicit criterion for a specialization of the Shintani cocycle to be $p$-adically interpolable. As a corollary, we recover the results of Deligne-Ribet,  Cassou Nogu\`es and Barsky on the construction of $p$-adic $L$-functions attached to totally real fields.}
\tableofcontents
\section{Introduction}
The goal of this paper is to construct cocycles on arithmetic subgroups of $\GL_n$, valued in spaces of $p$-adic pseudo-measures, which specialize to $p$-adic $L$-functions. Our starting point will be the \emph{Shintani zeta functions}, which generalize the Hurwitz partial zeta functions
\begin{equation*}
	\zeta_H([a+f\Z];s):=\sum_{\substack{n\in\N\\ n\equiv a~(f)}} \frac{1}{n^s}
\end{equation*}
to several dimensions. Given a lattice $L\subset\R^n$ (more generally, a finite linear combination of characteristic functions of affine lattices, e.g. a test function) and a simplicial cone $C\in\R^n_+$, the zeta function is defined by
\begin{equation*}
	\zeta_{SH}([L],C;s):=\sum_{v\in C} \frac{1}{N(v)^s}, ~~\Re(s)\gg0,
\end{equation*}
where $N(v)$ is the product of the coordinates. Shintani \cite{Shi76} showed these enjoy meromorphic continuation to the entire complex plane and computed their special values in terms of generalized Bernoulli polynomials. Moreover, Shintani showed how to decompose the $L$-functions of totally real fields into Shintani zeta functions. The explicit formulas for these values implies the rationality results of Siegel and Klingen provides the foundation for $p$-adic interpolation of these values by Cassou-Nogu\'es \cite{CN79} and Barsky \cite{Bar}. Our contribution will be a simple criterion, in terms of the cone $C$ and the ``test function $f$", for the special values of $\zeta_{Sh}([L],C;s)$ to be $p$-adically continuous.

Next, we turn to the \emph{Shintani cocycle}, constructed by Hill in \cite{Hi07}. This $n-1$ cocycle on $\GL_n$ takes values in a module of cones. Pairing a cone $C$ and a test function $f$ gives rise to a Shintani zeta function $\zeta_{SH}(f,C;s)$, and this pairing is obviously bilinear and $\GL_n$-equivariant. Our approach will be to fix a test function $f'$, away from $p$, and use this pairing to construct $p$-adic pseudo-measures corresponding to $\zeta_{SH}$. After restricting to a subgroup $\Gamma\subset\GL_n$ stabilizing $f$, we get a cocycle $\Phi_{f'}:\Gamma^n\rightarrow \widetilde{\cM}(\Z_p^n)$ valued in a space of pseudo-measures. We will describe which specializations (if any) yield $p$-adic measures in terms of our criterion. As an application, we will give a new construction of the $p$-adic $L$-functions of totally real fields.

Recently, Charollois and Dasgupta \cite{ChDa} have obtained similar results with Sczech's $\GL_n$ cocycle as part of a program to study the Gross-Stark units. In their work, they define an $\ell$-smoothed Sczech cocycle and deduce integrality results from explicit formulas in terms of Dedekind sums. As a consequence of these integrality results, they construct the $p$-adic measures corresponding to the zeta values of totally real fields. They have announced similar results for a version of the Shintani cocycle, but their techniques are substantially different from ours. Concurrently, Spiess \cite{Spi12} has constructed $p$-adic measures from the Shintani cocycle, adapting the argument of Cassou-Nogu\`es. Again, our techniques differ substantially. Rather than constructing measures from integrality results, we find the $p$-adic pseudo-measures as specializations of the Shintani cocycle, then show that these specializations are in fact measures via our elementary arguments. 
\subsection*{Acknowledgements}
First and foremost, the author wishes to thank his advisor, Glenn Stevens, for his countless contributions to this paper. This work was undertaken as part of the author's Ph.D. dissertation and would not have been possible without the generous support and constant encouragement from his advisor. I thank him for his patience, his wealth of ideas, and his careful reading of many drafts. Additionally, we wish to acknowledge the ancestry of this paper, which has its origins in Stevens's article ``The Eisenstein measure and real quadratic fields," \cite{Ste89}.

This project was motivated, in large part, by the works of S. Dasgupta and H. Darmon (see \cite{DDa06}, \cite{Das07}, \cite{Das08}) on Gross-Stark units. It is a pleasure to acknowledge the influence of their work, and to thank Dasgupta for his encouragement. Finally, I would like to thank J. Pottharst and R. Pollack for conversations which led to improvements in the exposition.

\section{Notation and Definitions}
Let $V$ be a finite dimensional $\Q$-vector space. We will always assume that we have picked a distinguished lattice $L\subset V$. In the case $V=\Q$, we let $L=\Z$.

For each prime $p$, we will write $V_p:=V\otimes_{\Q}\Q_p$, and $L_p:=L\otimes_{\Z}\Z_p$. We write $V_\R:=V\otimes_\Q \R$. If we equip $V_\R$ with a basis $e_1,\ldots,e_n$, then we will denote by $(V_\R)_+$ the positive orthant $\R_+e_1+\cdots \R_+ e_n$, where $\R_+=(0,\infty)$. We will suppress this choice of basis from our notation.

The group of test functions on $V$, denoted $\cS(V)$, is simply the $\Z$-module of Bruhat-Schwartz functions on the finite adeles $\bA_V^{(\infty)}$. Denote by $\cS(V_p)$ the space of step functions (locally constant, compact support) $f_p:V_p\rightarrow \Z$. For example, if $U\subset V_p$ is a compact open, we will write $[U]$ for the characteristic function of $U$. The group of test functions on $\Z$ is defined by
\begin{equation}
	\cS(V)=\bigotimes_{p}{}' \cS(V_p)
\end{equation}
where the restricted product means $f_p=[L_p]$ almost everywhere.

\begin{lem}
Equipping $V$ with the lattice topology, $\cS(V)$ is the space of locally constant functions with bounded support $f:V\rightarrow\Z$. In other words, a test function $f\in\cS(V)$ is a finite linear combination of characteristic functions of affine lattices.
\end{lem}

For each prime $p$, we will denote by $\cS(V^{(p)})$ the space
\begin{equation}
	\bigotimes_{\ell\neq p}{}' \cS(V_\ell),
\end{equation}
and refer to elements $f'\in\cS(V^{(p)})$ as a test functions \emph{away from $p$}

Our convention will be to let $\GL(V)$ act on $V$ on the left. By duality, this endows $\cS(V)$ with a \emph{right} $\GL(V)$ action, $(f|\gamma)(v):=f(\gamma v)$.

If $v_1,\ldots,v_r$, $r\leq n$, are linearly independent vectors in $V_\R$, we write $C^o(v_1,\ldots,v_r)$ for the set of all positive linear combinations $C^o(v_1,\ldots,v_r)=\{\sum a_i v_i | a_i\in\R_+\}$. $C(v_1,\ldots,v_r)$ will denote the closed cone $C(v_1,\ldots,v_r)=\{\sum a_i v_i | a_i\in\R_+\}$. In either case, we will call the rays in the directions $v_1,\ldots,v_r$ the \emph{extremal rays}. More generally, a \emph{simplicial} cone $C$ is a finite union of open cones (glued along boundaries). A \emph{pointed} cone is a cone that does not contain any lines. 

We will say a pointed open cone $C\subset V_\R$ is rational if it is generated by vectors $v_1,\ldots,v_r\in V$. More generally, a simplicial cone is rational if it us the union of rational open cones.

\section{Shintani's method}
\subsection{Shintani zeta functions}

If $C$ is a pointed simplicial cone in $V_\R$ and $f\in\cS(V)$ is a test function, the \emph{Shintani zeta function} $\zeta_{Sh}(f,C;s)$ is defined, for $\Real(s)\gg0$, as the sum
\begin{equation*}
	\zeta_{Sh}(f,C;s):=\sum_{v \in C} \frac{f(v)}{\Norm(v)^s}
\end{equation*}
where $\Norm(v)=e_1^*(v)\cdots e_n^*(v)$ is the product of the coordinates. One can show that the sum converges for $\Real(s)\gg0$ (see, for example, \cite{Hida}), and Shintani showed the these have meromorphic continuation to $s\in\C$. Moreover, the values $\zeta_{SH}(f,C;-k)$ can be expressed in terms of Bernoulli polynomials.

Shintani used these results to study the special values of Hecke $L$-functions of totally real fields. Let us suppose that  $F$ is a totally real field of degree $n$. The Hecke $L$-functions of $F$ decompose as sums of partial zeta functions, sometimes called \emph{ray class zeta functions}. If $\frak{f}$ is an integral ideal of $F$ and $\frak{a}$ is a fractional ideal relatively prime to $\frak{f}$, then the ray class zeta function for $[\frak{a}]_{\frak{f}\infty}$ is defined by
\begin{equation*}
	\zeta([\frak{a}]_{\frak{f}\infty},s):=\sum_{\substack{\frak{b}\subset\cO_F\\\frak{b}\sim_\frak{f} \frak{a}}} \frac{1}{\Norm\frak{b}^s} \text{	for} \Real(s)\gg0
\end{equation*}
where the sum is over all integral ideals representing $[\frak{a}]_{\frak{f}\infty}$ in the narrow ray class group. Two ideals $\frak{a}$ and $\frak{b}$ are equivalent in the narrow ray class group if and only if there exists a totally positive $\alpha\equiv1\pmod{\frak{f}}$ such that $(\alpha)=\frak{b}\frak{a}^{-1}$. Put 
\begin{equation*}
	E(\frak{f})=\{ u\in \cO_F^\times | u\gg 0 \text{ and } u\equiv 1\pmod{\frak{f}}\},
\end{equation*}
so that $(\alpha)\sim_\frak{f} (\beta)$ if and only if $\alpha\beta^{-1}\in E(\frak{f})$. We may rewrite the sum as
\begin{equation}\label{zetadecomp}
	\zeta([\frak{a}]_\frak{f},s)=\sum_{\substack{\beta\in (a+\frak{a}^{-1}\frak{f})/E(\frak{f})\\ ~~\beta\gg0}} \frac{1}{\Norm(\frak{a}\beta)^s}=\Norm\frak{a}^{-s}\sum_{\substack{\beta\in (a+\frak{a}^{-1}\frak{f})/E(\frak{f})\\ ~~\beta\gg0} } \frac{1}{\Norm\beta^s}
\end{equation}
where $a\in\frak{a}^{-1}$ is any fixed element congruent to $1\pmod{\frak{f}}$. If $\frak{a}$ is integral, then it suffices to take $a=1$. 

In order to interpret these ray class zeta functions as Shintani zeta functions, we embed $F$ in $\R^n$. Write $\tau_1,\ldots,\tau_n$ for the $n$ embeddings of $F$ into $\R$ and $F\hookrightarrow\R^n$ by $\alpha\mapsto(\tau_1(\alpha),\ldots,\tau_n(\alpha))$. The norm $N=e_1^*\cdots e_n^*$ on $\R^n$ extends the usual norm on $F$ to $\R^n$. Shintani's insight was to construct a fundamental domain for the action of $E(\frak{f})$ by decomposing $\R_+^n$ (where the totally positive elements live) into disjoint polyhedral cones. For example, if $F$ is a real quadratic field and $\varepsilon$ is a totally positive unit generating $E(\frak{f})$, then the polyhedral cone $C^o(1,\varepsilon)\cup C^o(1)$ forms a fundamental domain for the action of $\cO_F^\times$ (extended continuously to $\R^n$). More generally, Shintani proved
\begin{prop}[Proposition 4 of \cite{Shi76}]

Let $E\subset (\cO_F^\times)_+$ be a finite index subgroup of totally positive units. Then there exists a disjoint union of simplical cones, $C$, such that $\varepsilon C\cap C=\emptyset$ for all $\varepsilon\in E$ and
\[
	\R^n_+=\coprod_{\varepsilon\in E} \varepsilon C.
\]
\end{prop}
Such a collection of cones will be called a \emph{Shintani domain} for $E$. A Shintani domain lets us decompose the ray class zeta functions as
\begin{equation}\label{E:ray_is_shintani}
	\zeta([\frak{a}]_\frak{f},s)=(\Norm\frak{a})^{-s}\sum_{\alpha\in (a+\frak{a}^{-1}\frak{f} \cap C)} \frac{1}{\frak{\Norm{\alpha}}^s}=(\Norm\frak{a})^{-s}\zeta_{Sh}([a+\frak{a}^{-1}\frak{f}],C;s)
\end{equation}
reducing the study of special values of Hecke $L$-functions to the study of Shintani zeta functions.

\subsection{Special values}
Let us now suppose that $C$ is the ``open" cone $C=C^o(v_1,\ldots,v_r)$, with $v_1,\ldots,v_r\in V_{\R,+}$.  We remark that any cone can be written as a disjoint union of these ``open" cones, hence it suffices to treat only this case. Shintani's meromorphic continuation of
\begin{equation*}
	\zeta_{SH}(f,C;s)=\sum_{v\in C} \frac{f(v)}{N(v)^s}
\end{equation*}
generalizes Riemann's arguments for the meromorphic continuation of $\zeta(s)$. First, one expresses $\zeta_{SH}(f,C;s)$ as Mellin-transform by
\begin{align*}
	\sum_{v\in C} \frac{f(v)}{N(v)^s}=\sum_{v\in C} f(v)\frac{1}{\Gamma(s)^n} \int_{(0,\ldots,0)}^{(\infty,\ldots,\infty)} e^{-(e_1^*(v)x_1+\cdots e_n^*(v) x_n)} x^s\frac{dx}{x}
=\\
\frac{1}{\Gamma(s)^n}\int_{(0,\ldots,)}^{(\infty,\ldots,\infty)} \sum_{v\in C} f(v)e^{-(e_1^*(v)x_1+\cdots e_n^*(v) x_n)} x^s\frac{dx}{x}.
\end{align*}
To simplify notation, write $v\cdot x$ for $e_1^*(v)x_1+\cdots e_n^*(v)x_n)$. With the hypothesis that $f$ is rational with respect to $C$, there exist $a_1,\ldots,a_r\in\Q$ such that $f$ is periodic with respect to the lattice $a_1v_1\Z+\cdots +a_rv_r\Z$. After rescaling $v_1,\ldots,v_r$, we may assume $f$ is periodic with respect to the lattice $v_1\Z+\cdots v_r\Z$, then we can rewrite $ \sum_{v\in C} f(v)e^{-v\cdot x}$ as the ``rational function"
\begin{equation*}
	 \sum_{v\in C} f(v)e^{-v\cdot x}=\frac{1}{1-e^{-v_1\cdot x}}\cdots\frac{1}{1-e^{-v_r\cdot x}} \sum_{v\in\cP} f(v) e^{-v\cdot x},
\end{equation*}
where $\cP\subset C$ is fundamental domain for translation by $v_1\Z_{\geq 0}+\cdots v_n\Z_{\geq 0}$. Switching signs, write
\begin{equation*}
	G(x_1,\ldots,x_n)=\frac{1}{1-e^{v_1\cdot x}}\cdots\frac{1}{1-e^{v_r\cdot x}} \sum_{v\in\cP} f(v) e^{v\cdot x}.
\end{equation*}
The function $\frac{1}{e^z-1}$ has a simple pole at $z=0$ with residue $1$, so $G(x_1,\ldots,x_n)$ potentially has simple poles along the hyperplanes $v_1\cdot x=0,\ldots,v_r\cdot x=0$.

Next, one would like to find the Mellin-transform of $G(-x)$ as a term in the complex contour integral
\begin{equation*}
	(1-e^{2\pi i sn})\int_{(0,\ldots,)}^{(\infty,\ldots,\infty)} G(-x) x^s\frac{dx}{x}=\int_C G(-z) e^{(s-1)\log(z_1)+\cdots (s-1)\log(z_n)}  dz,
\end{equation*}
where $C$ is a product of keyhole contours $+\infty\rightarrow +\infty$ around $0$. However, when $n>1$, the poles of $G$ will intersect any sphere about the origin, hence our contour $C$, and we come to an impasse. Shintani managed to circumvent these problems by cleverly decomposing the domain of the Mellin transform. For details, we refer the reader to Shintani's original paper \cite{Shi76} or the notes of Greenberg and Dasgupta \cite{AWS11} for very readable accounts.

The following theorem is a reformulation of Proposition $1$ of \cite{Shi76}.
\begin{thm}[Shintani]
The function $\zeta_{SH}(f,C;s)$ has meromorphic continuation to the whole complex plane with at most a simple pole at $s=1$. Moreover, the special values are given by
\begin{equation}\label{Shintaniresidue}
\zeta_{SH}(f,C;-k)=\frac{1}{n k!^n}\left( \sum_{i=1}^n\Coeff(G(ux_1,ux_2,\ldots,ux_n)\mid_{x_i=1};u^{nk}x_2^k\cdots x_n^k)\right)
\end{equation}
where $\Coeff(F(x_1,\ldots,x_n), x_1^{k_1}\cdots x_n^{k_1})$ denotes the coefficient of $x_1^{k_1}\cdots x_n^{k_1}$ in the Laurent series of $F$ about the origin. 
\end{thm}
Note that $G(x_1,\ldots,x_n)$ is not necessarily a sum of monomials $x_1^{k_1}\cdots x_n^{k_n}$ near $0$ (consider $\frac{e^{x+y}}{e^{x+y}-1}=\frac{1}{x+y}\sum_{n,m\geq 0}B_{n+m}\frac{x^ny^m}{n!m!}$). However, it's not hard to see that $G(ux_1,\ldots,ux_n)|_{x_1=1}$ has a well-defined Laurent series in powers of $u,x_2,\ldots,x_n$. If $G$ happens to be holomorphic at the origin, then equation (\ref{Shintaniresidue}) simplifies to 
\begin{align}
\zeta_{SH}(f,C;-k)=\frac{1}{k!^n}\Coeff(G(x_1,\ldots,x_n), x_1^k\cdots x_n^k)\\
=\Coeff(\frac{\partial^{nk}}{\partial^k x_1 \cdots \partial^k x_n}G(x_1,\ldots,x_n);x_1^0\cdots x_n^0)\\
=\frac{\partial^{nk}}{\partial^k x_1 \cdots \partial^k x_n}G(x_1,\ldots,x_n)\mid_{x=0}\label{simpleShintani}
\end{align}

\section{Pseudo-measures and zeta values}
\subsection{Pseudo-measures}
Let $U\subset V_p$ be a compact open.
\begin{defn}
$\cC(U)$ is the $\Q_p$ vector space of continuous functions $f:U\rightarrow \Q_p$. This is a $\Q_p$-Banach space under the sup-norm, $|f|=\sup_{v\in U} |f(v)|_p$.
\end{defn}

\begin{defn}
A $p$-adic measure $\mu$ on $U$ is a continuous linear functional $\mu:\cC(U)\rightarrow\Q_p$. We write $\cM(U)$ for the $\Q_p$-Banach space $\Hom_{cts}(\cC(U),\Q_p)$.
\end{defn}

The fundamental example is the Dirac delta. For each $v\in U$,  $\delta_v\in\cM(U)$ is defined by $\delta_v(f):=f(v)$.

The convolution of two measures $\mu,\nu\in\cM(U)$ is defined by
\begin{equation*}
	(\mu\ast\nu)(f)=\int_U \left(\int_U f(v+w)d\nu(w)\right)d\mu(v).
\end{equation*}
Note that convolving by $\delta_v$ has the effect of translating by $v$: $(\mu\ast\delta_v)(f)=\int_U f(v+w)d\mu(w)$. If $U$ is a lattice, then $\cM(U)$ is a commutative ring under the convolution product, and is isomorphic to a power series ring, as we will recall shortly. In particular, $\cM(U)$ is a domain.

The space of pseudo-measures on $U$ is a subring of the field of fractions of $\cM(L_p)$ which, in some sense, accommodates the poles (at $s=1$) of the $p$-adic zeta functions we construct. Let $S\subset\cM(L_p)$ denote the multiplicative subset generated by the set $\{\delta_0-\delta_v : v\in L_p,v\neq 0\}$.

\begin{defn}
The space of $p$-adic pseudo-measures on $L_p$ is the localization
\begin{equation*}
	\widetilde{\cM}(L_p):=S^{-1}\cM(L_p).
\end{equation*}
\end{defn}
Note that this definition differs from some standard definitions, e.g. Coates (\cite{Co89}). 

\begin{prop}
The natural map $\cM(L_p)\rightarrow\widetilde{\cM}(L_p)$ is injective.
\end{prop}
\begin{proof}
This follows from the fact (due to Amice) that $\cM(L_p)$ is isomorphic to a subring of power series over $\Q_p$, and thus is an integral domain. We will recall this fact in greater detail in the following section.
\end{proof}

The fundamental example of a pseudo-measure is the classical Kubota-Leopoldt $p$-adic zeta function. If we take $V=\Q$, then there is a unique pseudo-measure $\xi\in\cM(\Z_p)$ satisfying $(\delta_0-\delta_1)\ast\xi=\delta_0$. This pseudo-measure interpolates (in a sense that can be made precise) the special values of the Riemann zeta function. After a brief foray into measures on $L_p$, we shall generalize this  to pseudo-measures associated to cones (Propostion \ref{P:zeta_values}).

\subsection{The Amice transform}
After choosing coordinates, all computations will reduce to the case of measures on $\Z_p^n$. The space of measures on $\Z_p^n$ can be explicitly described in terms of power series over $\Z_p$. We refer to \cite{Am64} or \cite{Co04} for proofs.
When $n=1$, Mahler's theorem gives an ON-basis of $\cC(\Z_p)$ via the generalized binomial coefficients
\[
	\binom{x}{k}:=\begin{cases}\frac{x(x-1)\cdots(x-k+1)}{k!} &\text{ if }k\geq 1;\\
				1 &\text{ if }k=0.\end{cases}
\]
This generalizes in a straightforward way to the case of several variables. 
\begin{prop}\label{Mahler}
The functions $\left\{\binom{x_1}{k_1}\cdots\binom{x_n}{k_n}\right\}$ form an ON-basis of $\cC(\Z_p^n)$. Concretely, every $f\in\cC(\Z_p^n)$ can be written uniquely as
\[
	f(x_1,\ldots,x_n)=\sum_{k_1,\ldots,k_n\geq0} a_{k_1,\ldots,k_n} \binom{x_1}{k_1}\cdots\binom{x_n}{k_n}
\]
with $|a_{k_1,\ldots,k_n}|_p\rightarrow 0$ and $||f||=\sup |a_{k_1,\ldots,k_n}|_p$.
\end{prop}
\begin{proof}
See, for example, the remarks following Proposition $7$, Section 7 in \cite{Am64}.
\end{proof}

Proposition \ref{Mahler} tells us a measure $\mu$ is uniquely determined by the moments $\mu\left(\binom{x_1}{k_1}\cdots\binom{x_n}{k_n}\right)$, and that $||\mu||=\sup_k\left|\binom{x}{k}\right|_p$.  To ease notation, our convention will be to write $x$ and $k$ for the vectors $x=(x_1,\ldots,x_n)$ and $k=(k_1,\ldots,k_n)$. We simply write $\binom{x}{k}=\binom{x_1}{k_1}\cdots\binom{x_n}{k_n}$ when no confusion may arise. 

These moments can be conveniently packaged into (rigid) analytic functions on an appropriate $p$-adic space. 

For each $r\in\R_{+}$, let us write $B(a,r)\subset\C_p$ for the open disc $B(1,r)=\{z\in\C_p: |z-a|_p< r\}$. We will write $\cB$ for the open polydisk $B(1,1)^n$, and denote by $(q_1,\ldots,q_n)$ parameters on $\cB$. The space of (rigid) analytic functions on $\cB$, defined over $\Q_p$, is denoted by $A_{\Q_p}(\cB)$. It is the space of power series, in $q_1-1,\ldots,q_n-1$ over $\Q_p$ with bounded coefficients. For example, if $\alpha=(\alpha_1,\ldots,\alpha_n)\in\Z_p^n$, we will write $q^{\alpha}$ for the function $q\mapsto q_1^{\alpha_1}\cdots q_n^{\alpha_n}$. By the binomial theorem, $q_i^{\alpha_i}$ is analytic in $q_i-1$, so $q^\alpha$ is analytic.

\begin{defn}
Associated to a measure $\mu\in\cM(\Z_p^n)$ is the rigid analytic function (on $\cB$)
\begin{align}
	\cA(\mu)(q_1,\ldots,q_n)= \int_{\Z_p^n} q^xd\mu(x)=\sum_{k\in\Z_{\geq0}^n} \left(\int_{\Z_p^n} \binom{x}{k}d\mu(x)\right)q^k\\
	=\sum_{k_1,\ldots,k_n\geq 0} \left(\int_{\Z_p^n} \binom{x_1}{k_1}\cdots\binom{x_n}{k_n}d\mu(x)\right)(q_1-1)^{k_1}\cdots (q_n-1)^{k_n},
\end{align}
called the Amice transform of $\mu$.
\end{defn}
The function $\cA(\mu)(q_1,\ldots,q_n)\in A_{\Q_p}(\cB)$ uniquely determines the measure $\mu$. What's more, the following theorem of Amice tells us each bounded power series corresponds to a measure:
\begin{thm}[Amice]
The map $\cA$ is an isomorphism of $\Q_p$-Banach algebra $\cM(\Z_p^n,\Q_p)$ and $\Z_p[[q_1-1,\ldots,q_n-1]]\otimes_{\Z_p}\Q_p$.
\end{thm}

\begin{rem}
As a map of commutative rings, this extends to a homomorphism 
\begin{equation*}
\cA:\widetilde{\cM}(\Z_p^n)\hookrightarrow \Frac(\Z_p[[q_1-1,\ldots,q_n-1]]).
\end{equation*}
\end{rem}

\subsection{Pseudomeasures from cones}

In this section, we construct pseudo-measures from the data of a test function $f'\in\cS(V^{(p)})$ and a pointed simplicial cone $C$ (which we think of as the component at $\infty$). Before getting to the construction, we illustrate the ideas with a simple example.

Suppose the vectors $v_1,\ldots,v_n$ form a basis of $V$, and consider the open cone $C=C^o(v_1,\ldots,v_n)$. Consider the compact open lattice $U=\Z_pv_1+\cdots +\Z_pv_n\subset V_p$. If $f'\in\cS(\bA_V^p)$ is a test function away from $p$, then $f'\otimes[U]$ is a test function on $V$. Consider, for a moment, the formal sum of measures on $U$:
\begin{equation}
	\sum_{v\in C} f'\otimes[U](v) \delta_v.
\end{equation}
The techniques of \S2.2 suggest that this represents a pseudo-measure on $U$. Indeed, there exists $a_1,\ldots,a_n\in\Q$ such that $f=f'\otimes[U]$ is periodic with respect to the lattice $\Z a_1v_1+\cdots\Z a_n v_n$. Since $[U]$ is periodic with respect to $\Z v_1+\cdots \Z v_n$, we may take $a_1,\ldots,a_n$ to be $p$-adic units. Then the formal sum represents the pseudo-measure
\begin{equation}
	\mu=\left(\frac{1}{1-\delta_{a_1v_1}}\right)\cdots\left(\frac{1}{1-\delta{a_nv_n}}\right)\sum_{v\in\cP} f(v) \delta_v,
\end{equation}
where $\cP$ is the half-open parallelepiped $\{\sum_{i=1}^n \lambda_i v_i : \lambda_i\in\Q, 0<\lambda_i\leq 1\}$. Note that the sum is actually measure on $U$, since it is a \emph{finite} sum of weighted Dirac measures on $U$.

This pseudo-measure, it turns out, encodes the special values of the Shintani zeta function $\zeta_{Sh}(C,f,s)$. It is useful to keep in mind the following heuristic, which we emphasize is \emph{not} mathematically meaningful (but can be made so!).

If $\mu$ is a measure, $k\geq 0$ an integer, and $N=e_1^*\cdots e_n^*:V\rightarrow\Q$ is our norm function, then $N^k$ is a continuous function on $L_p$ and
\begin{align}
	\int_U N^k d\mu &``=" \sum_{v\in C} f(v) \int_U N^k d\delta_v\\
	&``=" \sum_{v\in C} f(v) N(v)^k\\
	&``=" \zeta_{SH}(f,C;s)\mid_{s=-k}
\end{align}
Again, this reasoning is nonsense but the conclusion is true and follows from Shintani's formulas.

Now let us state a general version of the above construction.
\begin{prop}\label{P:pseudo-measure}
Let $f'\in\cS(V^{(p)})$ be a test function, $C\subset V_\R$ a pointed simplicial cone, and $U\subset L_p$ a compact open. Then there exists a unique pseudo-measure $\mu_{f',C,U}\in\widetilde{\cM}(L_p)$ with Amice transform
\begin{equation}
	\cA(\mu_{f',C,U})(q_1,\ldots,q_n) = \sum_{v\in C} f'\otimes[U](v)q^v
\end{equation}
in a neighborhood of $(q_1,\ldots,q_n)=(0,\ldots,0)$. We remark that the cone is not assumed to be $n$-dimensional, i.e. we may take $C=C^o(v_1,\ldots,v_r)$ with $r<n$.
\end{prop}
For notational simplicity, we will often write $\mu_{f',C}$ instead of $\mu_{f',C,L_p}$.

\begin{prop}\label{P:zeta_values}
Keeping the same hypothesis on $C,f'$, suppose furthermore that $\mu_{C,f'}$ is a measure, and that $C\subset (V_\R)_+$. Then
\begin{align*}
	\int_{U} N^k(v)d\mu_{C,f'}=\text{ the value at $s=-k$ of the analytic continuation of}\\
		\zeta(f'\otimes [U],C;s)=\sum_{v\in C}  \frac{f'\otimes [U](v)}{N(v)^s}
\end{align*}
\end{prop}
\begin{proof}
Fix $C,f'$ as above and write $\mu$ for the measure $\mu_{C,f'}$. Let $D_N: A_{\Q_p}(\cB)\rightarrow A_{\Q_p}(\cB)$ be the (invariant?) differential operator $D_Nq^v=N(v)q^v$. After rearranging a uniformly convergent series, we have
\begin{equation*}
	\int_U N^k(v)d\mu = \left.\left(\int_U D_N^kq^v d\mu\right)\right|_{q=1}= D_N^k\cA(\mu)\mid_{q=1}.
\end{equation*}
We make the change of variables $q_i=e^{x_i}$, so $q^v=e^{e_1^*(v) x_1+\cdots e_n^*(v)x_n}$. Under this change of variables, $D_N^k=\frac{\partial^{nk}}{\partial^k x_1 \cdots \partial^k x_n}$ and $\cA(\mu)$ becomes the function  $G(x_1,\ldots,x_n)$ represented by
\begin{equation}
\sum_{v\in C} f'\otimes[U] e^{e^*_1(v)x_1+\cdots+ e_n^*(v)x_n}.
\end{equation}
Since $\cA(\mu)$ is holomorphic at $q=1$, $G$ is holomorphic at $x_1,\ldots,x_n=0$. Shintani's theorem then implies
\begin{equation*}
	 D_N^k\cA(\mu)\mid_{q=1}=\frac{\partial^{nk}}{\partial^k x_1 \cdots \partial^k x_n}G(x_1,\ldots,x_n)|_{x_1,\ldots,x_n=0}=\zeta_{SH}(f'\otimes[U],C;-k).
\end{equation*}
\end{proof}

\subsection{Measure criteria}
In light of Proposition \ref{P:zeta_values} it is natural to ask, ``when is $\mu_{C,f'}$ a measure?" The main result of this section, and the technical heart of the paper, is an exact criterion for $\mu_{C,f'}$ to be a measure. Roughly speaking, this happens whenever the test function $f'$ has vanishing average in the directions of the extremal rays of $C$. To make precise this vague statement, we introduce some notation.

For each non-zero $w\in V$ and any $v\in V$, write $\pi_{v,w}:\cS(V_\ell)\rightarrow \cS(\Q_\ell)$ for the map which sends a test function $f\in\cS(V_\ell)$ to the function $\pi_{v,w}f:\Q_\ell\rightarrow \Z$ 
\begin{equation}
	(\pi_{v,w}f)(x):=f(v+x w ), \text{ for all } x\in\Q_\ell.
\end{equation}
A fortiori, $\pi_{v,w}f$ is indeed a test function on $\Q_\ell$. Similarly, we define $\pi_{v,w}:\cS(V^{(p)})\rightarrow\cS(\Q^{(p)})$ and $\pi_{v,w}:\cS(V)\rightarrow\cS(\Q)$.

The Haar measure on $\cS(V)$, normalized with respect to $L$, can be defined in two equivalent ways. First, for each $f\in\cS(V)$, there exist a lattice $L_f$ for which $f$ is periodic: $\forall\ell\in L_f$ and $v\in V$, $f(v+\ell)=f(v)$. One can define (see \cite{Ste89}) the global Haar measure $h_V$ by putting
\begin{equation}\label{H1}
	h_V(f):=\frac{1}{[L:L_f]}\sum_{v\in V/L_f} f(v).
\end{equation}
Since $f$ has bounded supported, the sum is finite, and it's easy to see that it is independent of choice of $L_f$. 

On the other hand, we have at each local component $V_\ell$, a local Haar measure $h_\ell$ normalized so that $h_\ell([L_\ell])=1$. Given $f=\bigotimes_{\ell} f_\ell \in\cS(V)$,  we can also define 
\begin{equation}\label{H2}
	h'_V(f):=\prod_{\ell} h_\ell(f_\ell)
\end{equation}
and extend to all of $\cS(V)$ by linearity.
\begin{lem}
The measures \ref{H1} and \ref{H2} both define the same Haar measure on $V$.
\end{lem}
\begin{proof}
See \cite{Ste89}, section 3.3.
\end{proof}
We define the Haar measure of test functions away from $p$ by defining $h^{(p)}(f'):=h_{V}(f'\otimes[L_p])$. If $f'$ is factorizable (i.e. $f'=\bigotimes_{\ell} f_\ell \in\cS(V))$, then $h^{(p)}(f')=\prod_{\ell\neq p} h_\ell(f')$.

\begin{defn}[{\bf Vanishing Hypothesis}]
Let $w\in V$ be a non-zero vector. We will say a test function $f'$ \emph{satisfies the {Vanishing Hypothesis} for $w$} if $h^{(p)}(\pi_{v,w} f')=0$ for all $v\in V$.
\end{defn}

We remark that $h^{(p)}(\pi_{v,w}f')$ depends only on $v \mod \langle w\rangle$ by the translation invariance of Haar. 
\begin{lem}\label{L:Haar_vanishing}
If a test function $f'\in\cS(V^{(p)})$ satisfies the vanishing hypothesis for $w$, then for all $U\subset V_p$ and $v\in V$,
\begin{equation}
	h_\Q( \pi_{v,w}(f'\otimes[U]))=0.
\end{equation}
\end{lem}
\begin{proof}
The vectors $v,w$ embed $\Q\hookrightarrow V$ and $\Q_p\hookrightarrow V_p$ via the inclusion $\lambda\mapsto v+\lambda w$. Let $W\subset \Q_p$ denote the projection of $U$ to the line $v+\lambda w$. Then $\pi_{v,w}(f'\otimes[U])=(\pi_{v,w}f' )\otimes (\pi_{v,w}[U])=(\pi_{v,w} f' )\otimes[W]$, and $h_V(\pi_{v,w} (f'\otimes[U]))=h^{(p)}(\pi_{v,w}f')h_p([W])=0$.
\end{proof}
One may interpret this lemma as saying a test function $f'$ satisfies the vanishing hypothesis for $w$ if the average value of $f'\otimes f_p$ is $0$ along all lines parallel to $w$, for all $f_p\in\cS(V_p)$. While this hypothesis may seem odd, it is in fact easy to verify in important cases. Indeed, we will show that the vanishing hypothesis is satisfied when $f'$ comes from the data of a ``smoothed" ray class zeta function of a totally real field. In the next section, we show how the construction of $p$-adic $L$-functions of totally real fields is a corollary of our main theorem.

We begin with a special case of the criteria,  from which we shall deduce the full result. Let $v_1,\ldots,v_r$ be linearly independent vectors in $L_p$, and put $C=C^o(v_1,\ldots,v_r)$. Extending $v_1,\ldots,v_r$ to a $\Q_p$-basis by $v_r,\ldots,v_n\in L_p$, put $U=\Z_pv_1+\cdots \Z_p v_r$. For each $v_0\in V$, consider the pseudo-measure $\mu_{f',C,v_0+U}\in\widetilde{\cM}(L_p)$. 

\begin{prop}\label{P:simple_case}
With $C=C^o(v_1,\ldots,v_r)$, a test function $f'\in\cS(V^{(p)})$ satisfies the vanishing hypothesis for $v_1,\ldots,v_r$ if and only if $\mu_{f',C,v_0+U}$ is a measure. 
\end{prop}
\begin{proof}
We only record the proof of the case $v_0=0$. The general case is virtually identical, modulo a few change of variables that we will indicate following the proof.

We know, from our construction of $\mu_{f',C,U}$, that
\begin{align}
	\cA(\mu_{f',C,U})=\sum_{v\in C} f'\otimes [U](v) q^v\label{E:at_zero}
\end{align}
in a neighborhood of $q=0$.
To realize this as a rational function, note that the test function $f'\otimes[U]$ is locally constant. That is to say, there exist $a_1,\ldots,a_r$ so that $f'\otimes[U]$ is periodic with respect to $a_1v_1,\ldots,a_rv_r$. Since $[U]$ is periodic with respect to $v_1,\ldots,v_r\in U$, we may take $a_1,\ldots,a_r$ to be $p$-adic units. Put
\begin{equation*}
\cP=\{ \sum_{i=1}^r \lambda_i v_r :\lambda_i\in\Q, 0<\lambda_i\leq a_i\}.
\end{equation*} 
For each $v\in C$, there exists a unique $v\in\cP$ and positive integers $n_1,\ldots,n_r$ such that $v=w+n_1a_1v_1+\cdots n_r a_r v_r$. Therefore the power series (\ref{E:at_zero}) represents the rational function (on $\cB$)
\begin{align*}
	\frac{1}{1-q^{a_1v_1}}\cdots\frac{1}{1-q^{a_rv_r}}\sum_{v\in \cP} f'\otimes[U](v) q^v,
\end{align*}
which we claim is holomorphic on $\cB$ if and only if $f'$ satisfies the vanishing hypothesis. To see this, we change coordinates, putting $(1+T_i)=q^{v_i}$ for $i\in \{1,\ldots, n\}$. The function becomes
\begin{equation}\label{E:main_term}
	\frac{1}{1-(1+T_1)^{a_1}}\cdots\frac{1}{1-(1+T_r)^{a_r}} \sum_{ v\in\cP} f'\otimes[U](v) (1+T_1)^{v_1^*(v)}\cdots(1+T_n)^{v_n^*(v)}.
\end{equation}
Since $a_1,\ldots,a_r$ are $p$-adic units, $(1-(1+T_i)^{a_i})/T_i$ is a unit in $\Z_p[[T_1,\ldots,T_n]]$. Thus, this function is holomorphic if and only if $T_1\cdots T_r$ divides 
\begin{equation}
	F(T_1,\ldots,T_n):=\sum_{ v\in\cP} f'\otimes[U](v) (1+T_1)^{v_1^*(v)}\cdots(1+T_n)^{v_n^*(v)},
\end{equation}
which is equivalent to each of $T_1,\ldots,T_r$ dividing $F$. 

~

\noindent{\bf Claim:} For each $i\in 1,\ldots, r$,  $T_i|F$ if and only if $f'$ satisfies the vanishing hypothesis for $v_i$.

~

For notational simplicity, we focus on the case $i=1$. It is easy to see that $T_1$ divides $F(T_1,\ldots,T_r)\in\Z_p[[T_1,\ldots,T_n]]$ if and only if $F(0,T_2,\ldots,T_n)=0$, i.e.
\begin{equation}
	F(0,T_2,\ldots,T_n)=\sum_{ v\in\cP} f'\otimes[U](v) (1)^{v_1^*(v)}\cdots(1+T_n)^{v_n^*(v)}=0.
\end{equation}
We carefully rearrange the sum as
\begin{equation}
	F(0,T_2,\ldots,T_n)=\sum_{v\in\cP\cap v_1^\perp}\left(\sum_{x\in(0,a_1]} f'\otimes[U](v+ xv_1)\right)(1+T_2)^{v_2^*(v)}\cdots(1+T_n)^{v_n^*(v)}.
\end{equation}
The coefficient in the parenthesis can be rewritten as
\begin{equation}\label{E:haar}
\sum_{x\in(0,a_1]} f'\otimes[U](v+ xv_1)=a_1 \frac{1}{[\Z:a_1\Z]} \sum_{x\in(0,a_1]} f'\otimes[U](v+ xv_1)=a_1 h_\Q(\pi_{v,v_1}(f'\otimes[U])).
\end{equation}
This shows that $F(0,T_2,\ldots,T_n)=0$ if and only if $h_\Q(\pi_{v,v_1}f'\otimes[U]) =0$ for all $v\in V$. Lemma \ref{L:Haar_vanishing} tells this is equivalent to the vanishing hypothesis for $v_1$, so $T_1$ divides $F$ if and only if $f'$ satisfies the vanishing hypothesis for $v_1$.. Similarly, we conclude $T_i$ divides $F$ if and only if $f'$ satisfies the vanishing hypothesis for $v_i$ giving the result.

More generally,
\begin{equation}
	\cA(\mu_{f',C,v_0+U})=\sum_{v\in C} f'\otimes[v_0+U](v)q^v=\sum_{v\in -v_0+C}f'(v_0+v)[U](v)q^{v_0+v}
\end{equation}
The vanishing hypothesis is translation invariant, so we may replace $f'(-v_0+v)$ with $f'(v)$. Thus we are reduced to showing
\begin{equation}
	q^{v_0}\sum_{v\in -v_0+C} f'\otimes[U](v) q^v
\end{equation}
represents a holomorphic function on $\cB$ when $f'$ satisfies the vanishing hypotheses for $v_1,\ldots,v_r$. Ignoring the $q^{v_0}$ term, this follows from the above argument, with $\cP$ replaced by $-v_0+\cP$. 
\end{proof}

\begin{thm}
Let $C$ be a simplicial cone with extremal rays $v_1,\ldots,v_r$. The pseudo-measure $\mu_{f',C}\in\widetilde{\cM}(L_p)$ is a measure if $f'$ satisfies the vanishing hypothesis for $v_1,\ldots,v_r$.
\end{thm}
\begin{proof}
Since $C(\lambda_1v_1,\ldots,\lambda_rv_r)=C(v_1,\ldots,v_r)$ for $\lambda_i\in\Q_+$, we may assume without loss of generality that $v_1,\ldots,v_r$ are contained in $L_p$. Furthermore, we may write $C$ as the disjoint union of open cones $C=\coprod_{j=1}^dC_j$, with each $C_j$ open and generated by a subset of $v_1,\ldots,v_r$. We pick $v_{r+1},\ldots,v_{n}$, also in $L_p$, so that $v_1,\ldots,v_n$ is a $\Q_p$-basis of $V_p$. Note that the lattice $U=\Z_pv_1+\cdots+\Z_p v_n$ is contained in the lattice $L_p$ with finite index. By Proposition \ref{P:pseudo-measure},
\begin{align*}
	\cA(\mu_{f',C}) &=\sum_{j=1}^d\sum_{v\in C_j} f'\otimes[L_p] (v) q^v\\
	&=\sum_{j=1}^d\sum_{v\in C_j} \sum_{x\in L_p/U} f'\otimes[x+U](v) q^v\\
	&=\sum_{j=1}^d\sum_{w\in L_p/U}\sum_{v\in C_j} f'\otimes[w+U](v)q^v,
\end{align*}
where the sum over $L_p/U$ is the sum over distinct cosets $x+V$. Again, by Proposition \ref{P:pseudo-measure}, this is the Amice transform of
\begin{equation}\label{E:decomposition}
	\sum_{j=1}^d\sum_{x\in L_p/U} \mu_{f',C_j,w+U}
\end{equation}
We conclude that $\mu_{f',C}$ is equal to (\ref{E:decomposition}) and is a measure when each summand is a measure. By Proposition \ref{P:simple_case}, we conclude $\mu_{f',C}$ is a measure when $f'$ satisfies the vanishing hypothesis for $v_1,\ldots,v_r$
\end{proof}

\section{The Shintani cocycle}
\subsection{Hill's construction}
We recall \S3 of \cite{Hi07}, slightly modifying Hill's conventions and construction . Hill's construction takes as input a choice of basis for $V$, so fix $\{w_1,\ldots,w_n\}$ a basis. 

Write $\cK^o_V$ for the abelian group of functions $V_\R\backslash\{0\}\rightarrow \Z$ generated by the characteristic function of rational open cones. We write $\cK_V$ for the group of functions $V_\R\rightarrow\Z$ whose restrictions to $V_\R\backslash\{0\}$ are in $\cK^o_V$. The group $\GL(V)$ acts on $\cK_V$ by
\begin{equation}\label{coneaction}
	(\gamma\cdot \kappa)(v)=\sign(\det\gamma) \kappa(\gamma^{-1} v).
\end{equation}
If $\kappa_1,\kappa_2$ are cone functions, then we will say $\kappa_1\leq \kappa_2$ if the support of $\kappa_1$ is contained in the support of $\kappa_2$.

The constant functions $V_\R\backslash \{0\}\rightarrow\Z$ form a submodule of $\cK_V$, and we write $\cL_V$ for the quotient $\cK_V/\Z$.

For example, if $v_1,\ldots,v_n$ are linearly independent vectors of $V$, the rational open cone $C^o(v_1,\ldots,v_n)$ is the set $\{ \sum_{i=1}^n \alpha_i v_i : \alpha_i\in \R_+\}$. Then the characteristic function of this open cone, denoted $[C^o(v_1,\ldots,v_n)]$, is an element of $\cK_V$.

\begin{defn}
If $\alpha_1w_1,\ldots,\alpha_nw_1$ are linearly independent, we will say $(\alpha_1,\ldots,\alpha_n)$ is non-degenerate. Degenerate will refer to the case that $\alpha_1w_1,\ldots,\alpha_nw_1$ are linearly dependent.
\end{defn}

Hill's cocycle is a $\GL(V)$-equivariant map $\sigma_{Hill}:\GL(V)^n\rightarrow\cK_V$ which, after quotienting out the constant functions, satisfies
\begin{equation}\label{cycle}
	\sum_{i=0}^n (-1)^n \sigma_{Hill}(\alpha_1,\ldots,\widehat{\alpha_i},\ldots,\alpha_n)=0.
\end{equation}
Naively, one might try to define a $\GL(V)$ cocycle by sending the tuple $(\alpha_1,\ldots,\alpha_n)$ to the cone function $[C^o(\alpha_1 w_1,\ldots,\alpha_n w_1)]$. However, one must decided what to do in the degenerate cases. Even after solving this problem, the resulting cocycle will no longer satisfying the cocycle condition (\ref{cycle}): the ``edges" of cones are missing, so they do not glue together. In the case of $V=\Q^2$, Solomon \cite{Sol98} solves these problems by giving the edges weight $1/2$. His cocycle (in Hill's language) is defined by
\begin{equation}\label{solomon}
	\sigma_{Solomon}(\alpha,\beta) = \sign\det(\alpha w_1,\beta w_1)\left([C^o(\alpha w_1,\beta w_1)]+\frac{1}{2}[C^o(\alpha w_1)]+\frac12[C^o(\beta w_1)]\right)
\end{equation}
with the convention that $\sign 0=0$. However, it's not clear how to extend this to higher dimensions. Solomon-Hu \cite{SolomonHu} define a cocycle on $\operatorname{PGL}_3(\Q)$, but their methods to not extend to higher dimension. Hill's construction, which we briefly recall, elegantly side-steps these problems. 

First, Hill notes that if $\{v_1,\ldots,v_n\}$ is a basis of $V$, then the cone function $[C^o(v_1,\ldots,v_n)]$ is given by
\begin{equation*}
	[C^o(v_1,\ldots,v_n)](w)=\begin{cases} 1 &\text{ if } v_1^*(w),\ldots,v_n^*(w)>0\\ 0 &\text{otherwise}\end{cases}
\end{equation*}
Next, Hill ``deforms" $\alpha_1 w_1,\ldots,\alpha_n w_1$ to a linearly independent set of vectors. Let $\F=\Q((\varepsilon_1))\cdots((\varepsilon_n))$. Every element of $f\in F$ can be expressed as a sum of monomials 
\begin{equation}\label{E:F_element}
	f=\sum_{\br=(r_1,\ldots,r_n)\in\Z^n} a_{\br} \varepsilon_1^{r_1}\cdots\varepsilon_{n}^{r_n}.
\end{equation}
Ordering the indices $\br\in\Z^n$ lexicographically, Hill defines the leading term of (a non-zero) $f$ to be the non-zero monomial $a_{\br}\varepsilon^\br$ for which $\br$ is smallest. For distinct $f,g\in \F$, Hill declares $f>g$ if the leading term of $f-g$ has positive coefficient, thus endowing $\F$ with the structure of an ordered field. Note that every positive power of $\varepsilon_j$ is smaller than every positive power of $\varepsilon_{j-1}$, and every positive power of $\varepsilon_1$ is smaller than every rational number.

Now consider the vector space $V_\F:=V\otimes_\Q \F$ over $\F$. For each $i\in\{1,\ldots,n\}$, define $b_i=w_1+\varepsilon_i w_2+\cdots +\varepsilon_i^{n-1} w_n$. This forms an $\F$-basis for $V_\F$ over $\F$, and in fact:

\begin{lem}[\cite{Hi07}, Lemma 1]\label{L:Hill_lemma_1}
For any $\alpha_1,\ldots,\alpha_n\in\GL(V)$ the vectors $\alpha_1 b_1,\ldots,\alpha_n b_n$ form a basis of $V\otimes_\Q \F$ over $\F$.
\end{lem}

Thus, for any $\alpha_1,\ldots,\alpha_n\in\GL(V)$, we have a natural cone function (on $V_\F$) by putting
\begin{equation*}
	[C^o(\alpha_1 b_1,\ldots,\alpha_n b_n)](w)=\begin{cases} 1 &\text{ if } (\alpha_1b_1)^*(w),\ldots,(\alpha_n b_n)^*(w)>0\\ 0 &\text{otherwise}\end{cases}
\end{equation*}

The key result is that this cone function on $V_\F$ restricts to a cone function on $V\subset V_\F$.

\begin{thm}[\cite{Hi07}, Theorem 1]\label{T:Hill_theorem_1}
The cone function $[C^o(\alpha_1b_1,\ldots,\alpha_n b_n)] :V_\F\rightarrow \Z$ restricts to a rational cone function $[C^o(\alpha_1b_1,\ldots,\alpha_n b_n] :V\rightarrow \Z$.
\end{thm}

\begin{defn}
Let $\sigma_{Hill}(\alpha_1,\ldots,\alpha_n)=\sign\det(\alpha_1b_1,\ldots,\alpha_n b_n)[C^o(\alpha_1 b_1,\ldots \alpha_n b_n)]|_{V}$.
\end{defn}

By Theorem \ref{T:Hill_theorem_1}, $\sigma_{Hill}$ is valued the module $\cK_V$. It is not hard to see that it is $\GL(V)$-equivariant, but moreover it satisfies the cocycle condition ({22}).

\begin{thm}[Hill]\label{T:Hill_description}
The map $\sigma_{Hill} : \GL(V)^n\rightarrow \cK_V$  is, modulo constant functions, an $n-1$ cocycle for $\GL(V)$. Moreover, if $(\alpha_1,\ldots,\alpha_n)$ is non-degenerate, there exists a simplicial cone 
\begin{equation}\label{E:cone_description}
	C^o(\alpha_1 w_1,\ldots, \alpha_n w_1)\subset C\subset C(\alpha_1w_1,\ldots\alpha_n w_1)
\end{equation}
such that $\sigma_{Hill}(\alpha_1,\ldots,\alpha_n)=\pm[C]$. 
\end{thm}
\begin{proof}
A calculation shows (\ref{E:cone_description})-- for details see \cite{Spi12}, Lemma 3.5.
\end{proof}

\subsection{The $\widetilde{\cM}$- valued cocycle}
Now fix $f'\in\cS(V^{(p)})$. Write $\Gamma\subset\GL(V)$ for the stabilizer of $f'\otimes[L_p]$. We apply our methods to Hill's cocycle, constructing a pseudo-measure valued cocycle and describe when its specializations are measures.

Let us write $\Gamma\subset\GL(V)$ for the stabilizer of $f'\otimes[L_p]$. Note that $\Gamma$ acts on $L_p$ (on the left), and hence by duality, $\Gamma$ has a right action on $\cC(L_p)$ and a left action on $\cM(L_p)$. This extends to a left action of $\Gamma$ on $\widetilde{\cM}(L_p)=\cM(L_p)\otimes_{\Z[L_p]}\Z[L_p]_S$ by acting on each term, endowing $\widetilde{\cM}(L_p)$ with the structure of a $\Z[\Gamma]$-module. 

If $\kappa\in \cK_V$ is a cone function, then it is a finite $\Z$-linear combination of characteristic functions of pointed open cones. Each open cone, with $f'$, gives us a pseudo-measure $\mu_{C,f'}$. After taking appropriate linear combinations, we have a pseudo-measure $\mu_{\kappa,f'}\in\widetilde{\cM}(L_p)$. In this way, we have a (well-defined!) map $\cK_V\rightarrow\widetilde{\cM}(L_p)$, and in fact:

\begin{lem}
The map $\cK_V\rightarrow\widetilde{\cM}(L_p)$ is a homomorphism of $\Z[\Gamma_{}]$-modules.
\end{lem}
\begin{proof}
This is an elementary calculation, and can be done by comparing power series at $q=0$. See also Solomon-Hu \cite{SolomonHu}, Lemma 2.1, or \cite{Hi07}, \S2.
\end{proof}

\begin{defn}
The $p$-adic Shintani cocycle attached to the data of $(V,L,f')$ is the composition
\begin{equation}
	\Phi_{f'}:\Gamma_{}^n\xrightarrow{\sigma_{Hill}}\cK_V\rightarrow\widetilde{\cM}(L_p)\rightarrow\widetilde{\cM}(L_p)/\Z\delta_0.
\end{equation}
It is an $n-1$ cocycle for $\Gamma$ valued in $\widetilde\cM(L_p)/\Z\delta_0$.
\end{defn}

Our main theorem states:

\begin{thm}\label{T:measure_valued}
Suppose $f'$ satisfies the vanishing hypothesis for $w_1$. Then $\Phi_{f'}(\alpha_1,\ldots,\alpha_n)$ is a measure on $L_p$ for all non-degenerate $(\alpha_1,\ldots,\alpha_n)\in \Gamma^n$. 
\end{thm}
Before proceeding the the proof, we record an elementary lemma.
\begin{lem}\label{L:gamma_action}
For all non-zero $w \in V$, $\gamma\in\Gamma_{f}$, and $v\in V$:
\begin{equation}
	\pi_{v,\gamma w } f(x)=f(v+x\gamma w)=f(\gamma (\gamma^{-1}v+x
	w))=\pi_{\gamma^{-1}v,w} f|\gamma=\pi_{\gamma^{-1}v,w} f.
\end{equation}
\end{lem}

Now we are ready to prove the main result.
\begin{proof}
If $f'$ satisfies the vanishing hypothesis for $w_1$, then by Lemma \ref{L:gamma_action} it satisfies the vanishing hypothesis for all $v\in\Gamma w_1$. If $(\alpha_1,\ldots,\alpha_n)\in\Gamma^n$ is non-degenerate, then Equation (\ref{E:cone_description}) implies $\sigma_{Hill}(\alpha_1,\ldots,\alpha_n)$ is (up to sign) the characteristic function of open cones $C_i$ generated by $\alpha_1w_1,\ldots,\alpha_nw_1$. The vanishing criterion implies $\mu_{f',C_i}$ is a measure, so
\begin{equation}
	\Phi_{f'}(\alpha_1,\ldots,\alpha_n)=\pm\sum \mu_{f',C_i}
\end{equation}
is a measure.
\end{proof}
\begin{cor}
Suppose $\dim_\Q(V)=2$, and $f'$ satisfies the vanishing hypothesis for $w_1$. Then $\Phi_{f'}$ is a measure-valued cocycle for $\Gamma$
\end{cor}
\begin{proof}
Thanks to Theorem \ref{T:measure_valued}, we only have to verify that $\Phi_{f'}(\alpha_1,\alpha_2)$ is a measure in the degenerate case. Since $\Gamma\subset\GL(L)$ is a finite index subgroup, and $\GL(L)\equiv \SL_2(\Z)$ acts transitively on $L\subset V$, we can find $\gamma\in\Gamma$ such that $\gamma w_1$ is not in the line spanned by $\alpha_1 w_1,\alpha_2 w_1$. The cocycle condition tells us
\begin{equation*}
	\Phi_{f'}(\alpha_1,\alpha_2)-\Phi_{f'}(\alpha_1,\gamma)+\Phi_{f'}(\alpha_2,\gamma)\equiv 0\pmod{\Z\delta_0}.
\end{equation*}
Our choice of $\gamma$ implies $\Phi_{f'}(\alpha_1,\gamma)$ and $\Phi_{f'}(\alpha_2,\gamma)$ are measures, again by Theorem \ref{T:measure_valued}. Thus $\Phi_{f'}(\alpha_1,\alpha_2)$ is a measure.
\end{proof}

While this proof does not generalize to higher dimension, we believe that the conclusion should hold. That is, we believe the Shintani cocycle $\Phi_{f'}$ should be \emph{measure-valued} whenever $f'$ satisfies the vanishing hypothesis for $w_1$. However, Hill's cocycle becomes unwieldy in higher dimensional degenerate cases and our methods depend on knowledge of the generators of the cones. Even though we cannot conclude all specializations are $p$-adic measures, all cases of arithmetic interest are non-degenerate and fit within the framework of our results. Of particular interest is the case $V=F$, a totally real field of degree $n$.

\section{$p$-adic $L$-functions}
Let $F$ be a totally real field, $\frak{f}$ is an integral ideal (prime to $p$), $\frak{c}\nmid\frak{f}$ a prime ideal of degree $1$, and $m$ a nonnegative integer. For all fractional ideals $\frak{a}$ prime to $\frak{f}$, define for $s\in\C$
\begin{equation*}
	\zeta^*([\frak{a}]_{\frak{f}p^m},s)=\sum_{\substack{0\neq \frak{b}\subset\cO_F\\ [\frak{b}]_{\frak{f}p^m}=[\frak{a}]_{\frak{f}p^m}\\ (\frak{b},p)=1}}\frac{1}{\Norm(\frak{b})^s}
\end{equation*}
and
\begin{equation*}
	\zeta^*_\frak{c}([\frak{a}]_{\frak{f}p^m},s):=\zeta([\frak{a}]_{\frak{f}p^m},s)-\Norm(\frak{c})^{1-s}\zeta^*([\frak{a}\frak{c}^{-1}]_{\frak{f}p^m},s).
\end{equation*}
Note that if $m=0$, $\zeta^*([\frak{a}]_{\frak{f}p^m},s)=\zeta([\frak{a}]_{\frak{f}p^m},s)$. By Cebotarev, we may assume, without loss of generality, that $\frak{a}$ is relatively prime to $p$ and $\frak{c}$.

Let $\cX$ denote weight space, the rigid analytic variety $\cX:=\Hom_{cts}(\Z_p^\times,\mathbb{G}_m)$. We embed $\Z\hookrightarrow \cX(\Q_p)$ by $k\mapsto (t\mapsto t^k)$ (note that we do not project $t$ to $1+p\Z_p$). For arbitrary elements $s\in\cX(\C_p)$, $t\in\Z_p^\times$, we will write $t^{s}$ for the image $s(t)$.
\begin{thm}[Deligne-Ribet, Cassou Nogu\`es, Barsky]\label{T:p-adic_L-functions}
There exists a $p$-adic analytic function $\zeta_{\ell,p}([\frak{a}]_{\frak f p^m},s)$, $s\in\cX(\C_p)$, such that
\begin{equation*}
	\zeta_{\frak{c},p}([\frak a]_{\frak f p^m},-k)=\zeta_\frak{c}^*([\frak{a}]_{\frak fp^m },-k)
\end{equation*}
for all integers $k\geq 0$.
\end{thm}
The theorem will follow by taking $V=F$ and considering the Schwartz function
\begin{equation}
	f'=\bigotimes_{q\nmid p\ell}[1+\frak{a}^{-1}\frak{f}\cO_{F,{q}}]\bigotimes\left([\cO_{F,{\ell}}]-\ell[\frak{c}\cO_{F,\ell}]\right).
\end{equation}
Let us write $\cO_{F,p}$ for the lattice $\cO_F\otimes_\Z\Z_p\subset F\otimes_\Q\Q_p$. Note that $f'\otimes [1+p^m\cO_{F,p}]=[1+\frak{a}^{-1}\frak{f}p^m\cO_F]-\ell[c+\frak{a}^{-1}\frak{f}p^m\frak{c}\cO_F]$, where $c\in\frak{c}$ is prime to $\frak{f}$ and is $\equiv 1\pmod{\frak{f}}$. In what follows, it will be convenient to take $c\in\Q$. First, we verify the vanishing hypothesis for $f'$:

\begin{lem}\label{L:hypothesis_check}
The test function $f'$ satisfies the vanishing hypothesis for $w_1=1$.
\end{lem}
\begin{proof}
Fix $\alpha\in F$. The projection $	\pi_{\alpha,1} f'\in\cS(\Q^{(p)})$ factors as
\begin{equation}
	\pi_{\alpha,1} f' = \bigotimes_{q\nmid p\ell )}\pi_{\alpha,1}[1+\frak{a}^{-1}\frak{f}\cO_{F,q}](\alpha+x)\bigotimes\pi_{\alpha,1}\left([\cO_{F,\ell}]-\ell[\frak{c}\cO_{F,\ell}]\right),
\end{equation}
and so it suffices to show $h_\ell(\pi_{\alpha,1}[\cO_{F,\ell}]-\pi_{\alpha,1}\ell[\frak{c}\cO_{F,\ell}])=0$. Since $\ell$ splits completely in $F$, we may choose coordinates identifying $\cO_{F,\ell}$ with $\Z_\ell^n$, and $\frak{c}\cO_{F,\ell}$ with $\ell\Z_\ell\times\Z_\ell^{n-1}$. Then, if $\alpha\in \cO_{F,\ell}$, 
\begin{equation}
	\pi_{\alpha,1}([\Z_\ell^n]-\ell[\ell\Z_\ell\times\Z_\ell^{n-1}])=[-a+\Z_\ell]-\ell[-a +\ell\Z_\ell]
\end{equation}
where $\alpha\equiv a\pmod{\frak{c}}$, which clearly has Haar measure $0$. If $\alpha\not\in \cO_{F,\ell}$, then the projection is $0$, which also has Haar measure $0$. Thus $f'$ satisfies the vanishing hypothesis for $1$.
\end{proof}

Since $E(\frak f \frak c)\subset\Gamma$, pairing our cocycle with non-degenerate elements of $H_{n-1}(E(\frak f\frak{c}),\Z)$ gives us measures, and by picking out the right units we can recover zeta values as moments of our measure. The exact element we need to pair our cocycle is provided by Lemme 2.2 of \cite{Col87}, but it is not a priori clear that this will give us the correct zeta values. The problem is that Hill's cocycle, a priori, does not evaluate to Shintani domains when the degree of the field is greater than $2$. However, Spiess has shown that Hill's construction does indeed recover Shintani domains:

\begin{prop}[Spiess]\label{P:is_domain}
Let $\eta \in \Z[E(\frak f\frak{c})^n]$ be a generator of $H_{n-1}(E(\frak f\frak{c}),\Z)\cong \Z$. Then the cone function $\sigma_{Hill}(\eta)$ is $\pm$ the characteristic function of a Shintani domain.
\end{prop}
\begin{proof}
This is Proposition 3.7 of \cite{Spi12}. 
\end{proof}

\begin{prop}
The pairing $\Phi_{f'}\cap\eta$ is a measure.
\end{prop}

\begin{proof}
Let $\varepsilon_1,\ldots,\varepsilon_{n-1}$ be fundamental units of $E(\frak{f}\frak{c})$. From Remark 2.1(c) of \cite{Spi12},  $\eta=\pm\sum_{\tau\in S_{n-1}} \sign(\tau) [\varepsilon_{\tau(1)}|\cdots|\varepsilon_{\tau(n-1)}]$, where $[\varepsilon_{\tau(1)}|\ldots|\varepsilon_{\tau(n-1)}]$ represents the cyclce $(1,\varepsilon_{\tau(1)},\varepsilon_{\tau(1)}\varepsilon_{\tau(2)},\ldots,\varepsilon_{\tau(1)}\cdots\varepsilon_{\tau(n-1)})\in\Z[\Gamma^n]$. By Lemma 2.1 of \cite{Col87}, this is non-degenerate. Using our Lemma \ref{L:hypothesis_check} and Theorem \ref{T:measure_valued}, we deduce that $\Phi_{f'}\cap\eta$ is a measure.
\end{proof}
Now we are ready to prove Theorem \ref{T:p-adic_L-functions}.
\begin{proof}
Fix $k\geq 0$ an integer, and let $\kappa=\sigma_{Hill}(\eta)$. By Proposition \ref{P:is_domain}, $\sigma_{Hill}(\eta)$ is $\pm$ the characteristic function of a Shintani domain for $E(\frak{f})$. In particular, $\kappa$ is supported on the positive orthant $\R_+^n$. Let $\mu$ be the measure $\mu=\pm\Phi_{f'}(\eta)$, where the sign is the sign of $\kappa$. By Proposition \ref{P:zeta_values}, the moments of $\mu$ are given by
\begin{align*}
	\int_{1+p^m\cO_{F,p}} \Norm(\alpha)^k d\mu(\alpha)=\zeta_{SH}(f'\otimes [1+p^m\cO_{F,p}],\kappa;-k)
\end{align*}
and
\begin{align*}
	\int_{\Norm^{-1}(\Z_p^\times)} \Norm(\alpha)^k d\mu(\alpha)=\zeta_{SH}(f'\otimes [\Norm^{-1}(\Z_p^\times)],\kappa;-k).
\end{align*}
By Equation \ref{E:ray_is_shintani},
\begin{equation}\label{E:interpolate1}
	\int_{1+p^m\cO_{F,p}} \Norm(\alpha)^k d\mu(\alpha)=\Norm(\frak{a})^k \zeta_\frak{c}([\frak{a}]_{\frak{f}p^m},-k)
\end{equation}
and
\begin{equation}\label{E:interpolate2}
	\int_{\Norm^{-1}(\Z_p^\times)} \Norm(\alpha)^k d\mu(\alpha)=\Norm(\frak{a})^k \zeta_\frak{c}^*([\frak{a}]_{\frak{f}},-k)
\end{equation}
If $m>0$, we define $\zeta_{\frak{c},p}([\frak{a}]_{\frak{f}p^m},s)$ to be the analytic function
\begin{equation}
	\zeta_{\frak{c},p}([\frak{a}]_{\frak{f}p^m},s):=\Norm(\frak{a})^{s}\int_{1+p^m\cO_{F,p}} N(\alpha)^{-s} d\mu(\alpha),
\end{equation}
where $\Norm(\alpha)^{-s}:=s(\Norm(\alpha)^{-1})$. If $m=0$,
\begin{equation}
	\zeta_{\frak{c},p}([\frak{a}]_{\frak{f}},s):=\Norm(\frak{a})^{s}\int_{\Norm^{-1}(\Z_p^\times)} N(\alpha)^{-s} d\mu(\alpha).
\end{equation}
By equations \ref{E:interpolate1} and \ref{E:interpolate2}, $\zeta_{\frak{c},p}$ has the correct interpolation property.
\end{proof}

\providecommand{\bysame}{\leavevmode\hbox to3em{\hrulefill}\thinspace}
\providecommand{\MR}{\relax\ifhmode\unskip\space\fi MR }
\providecommand{\MRhref}[2]{%
  \href{http://www.ams.org/mathscinet-getitem?mr=#1}{#2}
}
\providecommand{\href}[2]{#2}

\end{document}